\documentclass[pdflatex,sn-mathphys-ay]{sn-jnl}

\usepackage{graphicx}%
\usepackage{multirow}%
\usepackage{amsmath,amssymb,amsfonts}%
\usepackage{amsthm}%
\usepackage{mathrsfs}%
\usepackage[title]{appendix}%
\usepackage{xcolor}%
\usepackage{textcomp}%
\usepackage{manyfoot}%
\usepackage{booktabs}%
\usepackage{algorithm}%
\usepackage{algorithmicx}%
\usepackage{algpseudocode}%
\usepackage{listings}%

\theoremstyle{thmstyleone}%
\newtheorem{theorem}{Theorem}%
\newtheorem{proposition}[theorem]{Proposition}%
\newtheorem{lemma}[theorem]{Lemma}

\theoremstyle{thmstyletwo}%
\newtheorem{example}{Example}%
\newtheorem{remark}{Remark}%

\theoremstyle{thmstylethree}%
\newtheorem{definition}{Definition}%

\newcommand{\ZZ}{\mathbb{Z}}
\newcommand{\OO}{\mathcal{O}}
\newcommand{\LL}{\mathcal{L}}
\newcommand{\RR}{\mathcal{R}}
\newcommand{\NN}{\mathcal{N}}
\newcommand{\PP}{\mathcal{P}}

\newcommand{\gset}{(M,f_{L},f_{R},W_{L},W_{R})}

\newcommand{\pZ}{\mathbb{Z}_{\ge 0}}
\newcommand{\oZ}{\mathbb{Z}_{\ge 1}}
\newcommand{\seq}{\{\OO_S(n)\}_{n\in\pZ}}
\newcommand{\oseq}{\{\OO(n)\}_{n\in\pZ}}
\newcommand{\epsn}{\OO_S^{EP}}
\newcommand{\sn}{\OO_S^{SN}}

\begin{document}

\title[On Ending Partizan Subtraction Nim]{On Ending Partizan Subtraction Nim}

\author*[1]{\fnm{Hiyu} \sur{Inoue}}\email{hiyuuinoue@gmail.com}
\equalcont{These authors contributed equally to this work.}

\author[1]{\fnm{Shin-nosuke} \sur{Kadowaki}}
\equalcont{These authors contributed equally to this work.}

\author[1]{\fnm{Shun-ichi} \sur{Kimura}}
\equalcont{These authors contributed equally to this work.}

\author[1]{\fnm{Haruki} \sur{Wada}}
\equalcont{These authors contributed equally to this work.}

\affil*[1]{\orgdiv{Department of Mathematics}, \orgname{Hiroshima University}, \orgaddress{\street{Kagamiyama}, \city{Higashi-Hiroshima City}, \postcode{739-8526}, \state{Hiroshima}, \country{Japan}}}

\abstract{We consider Subtraction Nim, where two players have exactly the same options, but which is partizan in the sense that at the game ending, a partizan rule is applied for the decision of the winner. We consider the following example: Let $S$ be the set of removable numbers, which is a non-empty finite subset of positive integers greater than or equal to $2$, applied for both players Left and Right. At the end of the game, Left wins if the number of remaining tokens is even, and Right wins if the number of remaining tokens is odd. We computed the outcomes for many $S$, and found surprising phenomena that in most examples of $S$ (almost $98\%$ of some samples), the outcomes are $\mathcal{L}$-positions for all large enough $n$. In comparison, $\mathcal{R}$-positions appear only occasionally. The main theorem explains why this phenomenon occurs. We prove that $n+1$ and $n-1$ are $\mathcal{L}$-positions when $n$ is an $\mathcal{R}$-position. Similarly, $\mathcal{L}$-positions appear whenever $\mathcal{P}$-positions or $\mathcal{N}$-positions appear. Only $\mathcal{L}$-positions can last forever.}


\maketitle

\section{Introduction}\label{sec1}

An impartial game is defined as below by \citep{bib1}, and a similar definition is given in \citep{bib3}:
\begin{center}
    \textit{A game is impartial if both players have the same options from any position.}
\end{center}

In both references \citep{bib1} and \citep{bib3}, they prove following Theorem \ref{impartial outcome}, explicitly or implicitly assuming that an impartial game is either normal play or misère play:
\begin{theorem}\label{impartial outcome}
    If $G$ is a position of an impartial game then $G$ is in either $\NN$ or $\PP$.
\end{theorem}

In this paper, we give examples of such rulesets that the condition of conventional definition for impartial games is satisfied but the statement of Theorem \ref{impartial outcome} fails by making the rule to decide the winner partizan, and we call these new games Ending Partizan Games (Definition \ref{Ending Partizan Game}).

For example, consider $\LL\RR$-Subtraction Nim; let the set of removable numbers $S$ be a non-empty subset of positive integers greater than or equal to $2$, which is applied for both players Left and Right. At the end of the game, Left wins if the number of remaining tokens is even, and Right wins if the number of remaining tokens is odd. This game satisfies conventional Definition for impartial games, but Theorem \ref{impartial outcome} fails.

In $\LL\RR$-Subtraction Nim with $S=\{2,3\}$, Left has a winning strategy regardless of who moves first when there are $5k$ or $5k+2$ tokens where $k\in\pZ$. In contrast, Right has a winning strategy only when there is $1$ token. In the case $S=\{2,3\}$, Left holds an advantage over Right (Proposition \ref{ex1}).

Do there exist sets for which Right has an advantage? The answer is NO. On the contrary, in the case $S=\{2,3,6\}$, Left has a winning strategy if there are more than or equal to $8$ tokens (Proposition \ref{ex2}). Moreover, it turns out that such cases, in which Left consistently dominates, are rather common, if not all. Although it may initially seem that the rule ---Left wins if the number of remaining tokens is even, and Right wins if the number of remaining tokens is odd--- ensures fairness, a closer analysis reveals that this is not so.

Theorem \ref{main theorem} explains why this phenomenon occurs. We prove that $n+1$ and $n-1$ are $\LL$-positions when $n$ is an $\RR$-position. We also prove that in a sequence of outcomes $\seq$, $\PP$-positions consecutively appear $\min(S) - 1$ at most and both sides of them are $\LL$-positions, when $\PP$-position appears. In addition, assuming that $S$ is finite, $\NN$-positions consecutively appear $\max(S) - 1$ at most and both sides of them are $\LL$-positions, when $\NN$-position appears. Only $\LL$-positions can consecutively appear forever. Professor Urban Larsson informed us that there exist references such as \citep{bib2}, indicating that such examples are not unprecedented; however, this is the only instance we have identified. This paper is the first example of Ending Partizan Subtraction Nim.

Unless otherwise stated, we follow  \citep{bib1} for notation and terminology.

\section{Ending Partizan Subtraction Nim}\label{sec2}
\subsection{Definition of First Example}\label{sub2}
\begin{definition}[Combinatorial Game]
    A Combinatorial Game is a quintuple $\gset$ where $M$ is a set and $f_L,f_R,W_L,W_R$ are maps as below, which satisfies the finiteness condition $(\maltese)$;
    \begin{itemize}
        \item $f_L,f_R:M\to 2^{M}$.
        \item $W_L:\{G\in M\mid f_L(G)=\varnothing\}\to \{\text{Win}, \text{Lose}\}$.
        \item $W_R:\{G\in M\mid f_R(G)=\varnothing\}\to \{\text{Win}, \text{Lose}\}$.
    \end{itemize}
    $(\maltese)$ There does NOT exist any INFINITE sequence;
    \begin{equation*}
        m_0,m_1,m_2,m_3,\dots
    \end{equation*}
    where for each $i\in\pZ, m_i \in M$ and $m_{i+1}\in f_L(m_i)\cup f_R(m_i)$.
\end{definition}

The game is played as follows; starting from the initial position $m_0$, say Left is the first player, Left chooses $m_1\in f_L(m_0)$, and Right chooses $m_2\in f_R(m_1)$ alternatingly, until the player cannot choose next move, say $f_R(m_{2n+1})=\varnothing$, and the winner is decided by $W_R(m_{2n+1})$, if $W_R(m_{2n+1})=\text{Win}$ then Right wins, and otherwise, Left wins. If the game ends with $f_L(m_{2n})=\varnothing$, we decide the winner similarly: If $W_L(m_{2n})=\text{Win}$, Left wins, and if $W_L(m_{2n})=\text{Lose}$, Right wins. By the finiteness condition $(\maltese)$, it is guaranteed that the game ends after finitely many steps. If $W_R$ and $W_L$ are constant maps to Lose, then the rule is the normal play, and if both are constant maps to Win, then it is the misère play.

We propose our definition of impartial game as follows so that Theorem \ref{impartial outcome} holds without any implicit assumptions.
\begin{definition}[Impartial game]\label{our impartial}
    A combinatorial game $\gset$ is impartial game if $f_L=f_R$ and $W_L=W_R$.
\end{definition}

\begin{definition}[Outcome]
    Let $\gset$ be a combinatorial game.

    For $m\in M$, $\OO(m)=\LL$ if Left has a winning strategy regardless of whether Left is the first player, starting from $m$. $\OO(m)=\RR$ if Right has a winning strategy from $m$. $\OO(m)=\NN$ if the next (namely the first) player has a winning strategy regardless of whether the next player is Left, starting from $m$. $\OO(m)=\PP$ if the previous (namely the second) player has a winning strategy from $m$.

    By abuse of notation, we also denote the subset $\OO^{-1}(\LL)\subset M$ as $\LL$ etc.
\end{definition}

\begin{definition}[Ending Partizan Game]\label{Ending Partizan Game}
    A combinatorial game $\gset$ is Ending Partizan Game if $f_L=f_R$. We represent an Ending Partizan Game as $(M, f, W)$, where $f=f_L(=f_R)$ and $W:\{m\in M\mid f(m)=\varnothing\}\to \{\LL,\RR,\NN,\PP\}$,
    \begin{equation*}
        W(m) =
        \left\{
            \begin{array}{ll}
                \LL & \text{if } W_L(m)=\text{Win and }W_R(m)=\text{Lose}, \\
                \RR & \text{if } W_L(m)=\text{Lose and }W_R(m)=\text{Win}, \\ 
                \PP &\text{if } W_L(m)=\text{Lose and }W_R(m)=\text{Lose},\\
                \NN & \text{if } W_L(m)=\text{Win and }W_R(m)=\text{Win}.
            \end{array}
        \right.
    \end{equation*}
\end{definition}

\begin{proposition}\label{outcomes}
    Let $(M, f, W)$ be an Ending Partizan Game.

    For $m\in M$ where $f(m)=\varnothing$, then 
    \begin{equation*}
        \OO(m) = W(m).
    \end{equation*}

    For $m\in M$ where $f(m)\neq \varnothing$, then
    \begin{equation*}
        \OO(m) =
        \left\{
            \begin{array}{ll}
                \LL & \text{if and only if } \{\LL\}\subset \OO(f(m))\subset\{\LL,\NN\}, \\
                \RR & \text{if and only if } \{\RR\}\subset \OO(f(m))\subset\{\RR,\NN\}, \\ 
                \PP & \text{if and only if } \{\NN\}=\OO(f(m)), \\
                \NN & \text{if and only if } (\{\PP\}\subset\OO(f(m)) \text{ or } \{\LL,\RR\}\subset\OO(f(m))).
            \end{array}
        \right.
    \end{equation*}
\end{proposition}

\begin{proof}
    The statement trivially holds for any $m \in M$ such that $f(m) = \emptyset$ by definition.

    Now, consider $m$ such that $f(m) \neq \varnothing$.

    \begin{itemize}
        \item ( $\{\LL\}\subset \OO(f(m))\subset\{\LL,\NN\}$ ) If Left moves first, there exists some $m' \in f(m)$ such that $m'\in \LL$ since $\{\LL\} \subset \OO(f(m))$. Thus, Left can win as the first player. If Right moves first, all possible moves lead to $\LL$-positions or $\NN$-positions since $\OO(f(m)) \subset \{\LL, \NN\}$. Therefore, Right cannot win as the first player. Hence, $m \in \LL$.
        
        We can similarly show that if $\{\RR\} \subset \OO(f(m)) \subset \{\RR, \NN\}$, then $m \in \RR$.

        \item ( $\OO(f(m)) = \{\NN\}$ ) In this case, $m' \in \NN$ for any $m' \in f(m)$. Therefore, the first player cannot win. This implies that the second player has a winning strategy, so $m \in \PP$.

        \item ( $\{\PP\} \subset \OO(f(m))$ ) In this case, there exists some $m' \in f(m)$ such that $m' \in \PP$. Thus, the first player can win by getting to $m'$. Therefore, $m \in \NN$.
        
        \item ( $\{\LL, \RR\} \subset \OO(f(m))$ ) There exists some $m' \in f(m)$ such that $m' \in \LL$. Thus, Left can win by getting to $m' \in \LL$. We can similarly show that Right can win as the first player. Therefore, $m \in \NN$.
    \end{itemize}
    Noticing that exactly one of the conditions holds, the argument above guarantees the converse.
\end{proof}

\begin{definition}[$\LL\RR$-Subtraction Nim]\label{Ending Partizan Subtraction Nim}
    Let $S$ be a non-empty subset of positive integers greater than or equal to $2$. The $\LL\RR$-Subtraction Nim with $S$ is given in the next conditions:
    \begin{itemize}
        \item $M= \pZ$.
        \item $f:M\to 2^{M}, m\mapsto \{m-s\mid s\in S \text{ and } m-s\ge 0\}.$
        \item $W:\{m\in M\mid f(m)=\varnothing \}\to \{\LL,\RR,\NN,\PP\}, m \mapsto
        \left\{
            \begin{array}{ll}
                \LL & \text{if } m \text{ is even}, \\
                \RR & \text{if } m \text{ is odd}. 
            \end{array}
        \right.$
    \end{itemize}
    We call $S$ a removable set.
\end{definition}

In what follows, we write $\OO_S(n)$ for the outcome of $n$ tokens in $\LL\RR$-Subtraction Nim with $S$ a set of removable numbers. From now on, we call the sequence $\seq$ an outcome sequence.
\begin{proposition}\label{ex1}
    In $\LL\RR$-Subtraction Nim with $S=\{2,3\}$, the outcomes are:
    \begin{equation*}
        \OO_S(n)=
        \left\{
            \begin{array}{ll}
                \LL & \text{if } n\in \{5m+2,5m\mid m\in\pZ\}, \\
                \RR & \text{if } n=1, \\
                \NN & \text{if } n\in \{5m+3,5m+4\mid m\in \pZ\},\\
                \PP & \text{if } n\in \{5m+1\mid m\in \oZ\}.
            \end{array}
        \right.
    \end{equation*}
\end{proposition}

We write the sequence of outcomes horizontally.

\begin{table}[h]
\caption{The outcome sequence of $\LL\RR$-Subtraction Nim with $S=\{2,3\}$}%
\begin{tabular}{@{}lllllllllllllllll@{}}
\toprule
$n$ & $0$ & $1$ & $2$ & $3$ & $4$ & $5$ & $6$ & $7$ & $8$ & $9$ & $10$ & $11$ & $12$ & $13$ & $14$ & $\dots$\\
\midrule
$\OO_S(n)$  & $\LL$ & $\RR$ & $\LL$ & $\NN$ & $\NN$ & $\LL$ & $\PP$ & $\LL$ & $\NN$ & $\NN$ & $\LL$ & $\PP$ & $\LL$ & $\NN$ & $\NN$ & $\dots$\\
\botrule
\end{tabular}
\end{table}

\begin{proof}
    We prove the Proposition by induction on $m$. When $m=0$, $0\in \LL$ and $1\in \RR$ follow from the definition of the game. For the case of $2$ tokens, any move by the first player leads to $0$ token. Therefore $2\in \LL$. For the case of $3$ tokens, if Left moves first, Left can win by taking $3$ tokens; if Right moves first, Right can win by taking $2$ tokens. Therefore, $3\in \NN$. For the case of $4$ tokens, if Left moves first, Left can win by taking $2$ tokens; if Right moves first, Right can win by taking $3$ tokens. Therefore, $4\in \NN$. When $m\ge 1$, we may proceed by induction on $m$, assuming that Proposition \ref{ex1} holds for smaller $m$ except for $6$. As $f(6)=\{3, 4\}\subset \NN, 6\in \PP$.
    
    \begin{align*}
        \{\LL\}\subset f(5m) &= \{5(m-1)+2, 5(m-1)+3\}\subset \{\LL, \NN\}. \\
        f(5m+1) &= \{5(m-1)+3, 5(m-1)+4\}= \{ \NN\}. \\
        \{\LL\} \subset f(5m+2) &= \{5(m-1)+4, 5m\}\subset \{\LL, \NN\}.\\
        f(5m+3) &\ni 5m+1 \text{ with } \OO_S(5m+1)=\PP. \\
        f(5m+4) &\ni 5m+1 \text{ with } \OO_S(5m+1)=\PP. 
    \end{align*}
\end{proof}
\begin{proposition}\label{ex2}
    In $\LL\RR$-Subtraction Nim with $S=\{2,3,6\}$, the outcomes are:
    \begin{equation*}
        \OO_S(n)=
        \left\{
            \begin{array}{ll}
                \LL & \text{if } n=0,2,5,6 \text{ or } n\ge 8, \\
                \RR & \text{if } n=1, \\
                \NN & \text{if } n=3,4,7.
            \end{array}
        \right.
    \end{equation*}
\end{proposition}

The sequence of outcomes is below,

\begin{table}[h]
\caption{The outcome sequence of $\LL\RR$-Subtraction Nim with $S=\{2,3,6\}$}%
\begin{tabular}{@{}lllllllllllllllll@{}}
\toprule
$n$ & $0$ & $1$ & $2$ & $3$ & $4$ & $5$ & $6$ & $7$ & $8$ & $9$ & $10$ & $11$ & $12$ & $13$ & $14$ & $\dots$\\
\midrule
$\OO_S(n)$  & $\LL$ & $\RR$ & $\LL$ & $\NN$ & $\NN$ & $\LL$ & $\LL$ & $\NN$ & $\LL$ & $\LL$ & $\LL$ & $\LL$ & $\LL$ & $\LL$ & $\LL$ &  $\dots$ \\
\bottomrule
\end{tabular}
\end{table}

\begin{proof}
    We begin by calculating the outcome $\OO_S(n)$ for all integers $n\in \{0,1,\dots, 7\}$. For $n\le 5$, the outcomes are the same as those in Proposition \ref{ex1}. For the case of $n=6$, Left can win from $n$ by moving to $n-6=0\in \LL$. Any moves by Right lead to $n-2=4\in \NN$, $n-3=3\in \NN$ and $n-6=0\in \LL$ and Right loses, hence $6\in \LL$. Consider the case $n=7$. Left can win from $n$ by moving to $n-2=5\in \LL$. Right can win from $n$ by moving to $n-6=1\in \LL$. Therefore, $7\in\NN$.

    Next, suppose $n\ge 8$. Any moves from $n$ leads to $\LL$-positions or $\NN$-positions, so $n\in \LL$ if there exists at least one $\LL$-position among its options based on Proposition \ref{outcomes}.

    \[8\to 2\in \LL, 9\to 6\in \LL, 10\to 8\in \LL, 11\to 5\in \LL, 12\to 6\in \LL, 13\to 10\in \LL\]
    and $n-6\in \LL$ where $n\ge 14$. Thus $\OO_S(n)=\LL$.
\end{proof}
\subsection{Main Result}
We computed the outcomes for many $S$ (to be precise, $116794921$ cases out of $134217728$, almost $87\%$, examples of $S$ with the minimum of $S$ just $2$ and the maximum of $S$ less than or equal to $29$, also $131500955$ cases out of $134217728$, almost $98\%$, examples of $S$ with the minimum of $S$ just $3$ and the maximum of $S$ less than or equal to $30$ ), and found that the outcomes are $\LL$-positions for all large enough $n$. 

Theorem \ref{main theorem} and Lemma \ref{main lem} explain why these phenomena occur.

\begin{theorem}\label{main theorem}
    In $\LL\RR$-Subtraction Nim with $S\subset \mathbb{Z}_{\ge2}$, the following statements hold:
    \begin{enumerate}
        \item[(1)] If Right has a winning strategy from $n$, Left has a winning strategy from $n\pm 1$. That is,
        \begin{equation*}
            \OO_S(n)=\RR \implies \OO_S(n-1)=\OO_S(n+1)=\LL.
        \end{equation*}
        \item[(2)] If the previous player has a winning strategy in $n$, Left can win as the second player in $n\pm 1$. That is,
        \begin{equation*}
            \OO_S(n)=\PP \implies \OO_S(n-1), \OO_S(n+1)\in \{\LL,\PP\}.
        \end{equation*}
        \item[(3)] If the next player has a winning strategy in $n$, Left can win as the first player in $n\pm 1$. That is,
        \begin{equation*}
            \OO_S(n)=\NN \implies \OO_S(n-1), \OO_S(n+1) \in \{\LL,\NN\}.
        \end{equation*}
        \item[(4)] When $\PP$-position appears in the sequence of outcomes $\seq$, $\PP$-positions consecutively appear at most $\min(S)-1$ times and both sides of them are $\LL$-positions.
        \item[(5)] When $\NN$-position appears in the sequence of outcomes $\seq$, $\NN$-positions consecutively appear and the left side is $\LL$-position. Moreover, the right side is $\LL$-position when $\NN$-positions do not last forever. In particular, if $S$ is a finite set, $\NN$-positions consecutively appear at most $\max(S)-1$ times and both sides of them are $\LL$-positions.
    \end{enumerate}
\end{theorem}

To prove this Theorem, we first prepare the following Proposition. The basic idea to prove Theorem \ref{main theorem} is a strategy stealing, as follows: If Right can win from $n$ to finish at an odd number $k<\min(S)$, then Left can win from $n\pm 1$ to finish at the even numbers $k\pm 1$ by the same strategy, except for three exceptional cases, namely $k=0$, $k=\min(S)$, and $n+1\in S$.
\begin{lemma}\label{main lem}
    In $\LL\RR$-Subtraction Nim with $S\subset \mathbb{Z}_{\ge2}$, the following holds:
        \begin{enumerate}
            \item[(1)] When Right can win as the first player from $n$, Left can win as the first player from $n\pm1$. In other words, 
            \begin{equation*}
                \OO_S(n)\in\{\RR,\NN\} \implies \OO_S(n-1), \OO_S(n+1)\in\{\LL,\NN\}.
            \end{equation*}
            \item[(2)] When Right can win as the second player from $n$, Left can win as the second player from $n\pm1$. In other words, 
            \begin{equation*}
                \OO_S(n)\in\{\RR,\PP\} \implies \OO_S(n-1), \OO_S(n+1)\in\{\LL,\PP\}.
            \end{equation*}
        \end{enumerate}
\end{lemma}

\begin{proof}
    We prove both $(1)$ and $(2)$ simultaneously by induction on $n$. Let $s = \min(S)$.

    Consider the case $n<s$. When Right can win from $n<s$, then $n$ is odd and $n+1$ and $n-1$ are even, so $\OO_S(n+1)=\OO_S(n-1)=\LL$ by the definition of $\LL\RR$-Subtraction Nim except for the case $n=s-1$. However, when $n=s-1$, regardless of whether Left or Right moves first, the only option in $s$ is to take $s$ tokens, which results in a win for Left, so $\OO_S(n+1)=\LL$. Therefore, both $(1)$ and $(2)$ hold for $n<s$.

    Now we assume that both $(1)$ and $(2)$ hold for all non-negative integers less than or equal to $n-1$.

    For proof of $(1)$, suppose that in the position with $n$ tokens, Right can win as the first player. There exists $s_1\in S$ such that Right can win as the second player in $n-s_1$ tokens. Since $\OO_S(n-s_1)\in\{\RR,\PP\}$ and $\OO_S(0)=\LL$, it follows that $s_1 < n$, and hence $s_1\le n-1$. By the inductive hypothesis $(2)$, Left can win as the second player in $(n+1)-s_1$ and $(n-1)-s_1$, so in the position with $n+1$ (and $n-1$) tokens, Left can win as the first player by taking $s_1\in S$ tokens.

    For proof of $(2)$, suppose that in the position with $n$ tokens, Right can win as the second player. Then, for any $s_2\in S$ with $s_2\le n$, Right can win as the first player in $n-s_2$. By the inductive hypothesis $(1)$, Left can win as the first player in $(n+1)-s_2$ and $(n-1)-s_2$, so even if Right takes $s_2$ tokens from $n\pm 1$, Left can still win as the second player. The only remaining case we need to consider is the case when $n+1\in S$, we have to consider the situation in which Right takes all $n+1$ tokens from $n+1$. Since the result with $0$ token is win for Left, which proves that Left wins as the second player from $n+1$ and $n-1$.

    Hence, the inductive step is complete, and the Proposition is established.
\end{proof}

\begin{proof}[Proof of Theorem \ref{main theorem}]
    Regarding $(1)$, assuming that $\OO_S(n)=\RR$, then $\OO_S(n+1), \OO_S(n-1)\in\{\LL,\NN\}$ and $\OO_S(n+1), \OO_S(n-1)\in \{\LL,\PP\}$ by Lemma \ref{main lem} ($1$) and $(2)$. Thus $\OO_S(n+1), \OO_S(n-1)\in \LL$. $(2)$ and $(3)$ are special cases of Lemma \ref{main lem} $(1)$ and $(2)$. Regarding $(4)$, by $(2)$, we need only to show that $\PP$-positions appear consecutively at most $\min(S)-1$ times. When $\OO_S(n)=\PP$, then we can play from $n+\min(S)$ to $n$, so Proposition \ref{outcomes} implies that $\OO_S(n+\min(S))=\NN$, then by $(3)$, $\OO_S(n+\min(S)-1)\neq \PP$, hence consecutive appearance of $\PP$ is at most $\min(S)-1$. Regarding $(5)$, by $(3)$, we need only to show that $\NN$-positions appear consecutively at most $\max(S)-1$ times if $S$ is a finite set. Assume that $S$ is a finite set. If $\OO_S(n+1)=\dots =\OO_S(n+\max(S)-1)=\NN$, then any moves from $n+\max(S)+1$ lead to $\NN$-positions, so Proposition \ref{outcomes} implies that $\OO_S(n+\max(S)+1)=\PP$, then by $(2)$, $\OO_S(n+\min(S))\neq \NN$, hence consecutive appearance of $\NN$ is at most $\max(S)-1$.
\end{proof}

\begin{remark}
    Proposition \ref{main lem} does not hold if the roles of Left and Right are interchanged; this asymmetry is the reason for the bias toward Left. 
    
    For example, let $S=\{2,3\}$. As an example like (1), consider the case where $\OO_S(3)=\NN$. In this situation, Left wins if she is the first player. Furthermore, since $\OO_S(2)=\LL$, Left also wins when moving first. This occurs because Right is unable to remove three tokens from a heap that only contains two.

    As an example like (2), consider $\OO_S(2)=\LL$. In this case, Left has a winning strategy as the second player. However, since $\OO_S(3)=\NN$, Right cannot win as the second player. This is because Left can remove three tokens, leading to Right's defeat.
\end{remark}
\begin{remark}
    In Theorem \ref{main theorem} $(5)$, the finiteness assumption is essential. Suppose $S=\ZZ_{\ge 2}$. For any $n\in \ZZ_{\ge 3}$, from $n$ tokens, Left can win as the first player by taking $n$ tokens and Right can win as the first player by taking $n-1$ tokens, so $\OO_S(n)=\NN$ and $\NN$-positions last forever.
\end{remark}
Using Theorem \ref{main theorem}, in $\LL\RR$-Subtraction Nim with a finite $S$, the following statements on a sequence of outcomes $\seq$ hold:
\begin{itemize}
    \item If $\RR$-position appears, the both sides of it are $\LL$-positions. 
    \item $\PP$-positions appear consecutively at most $\min(S)-1$, and both sides of them are $\LL$-positions.
    \item $\NN$-positions appear consecutively at most $\max(S)-1$, and both sides of them are $\LL$-positions.
    \item No such restrictions for $\LL$-position.
\end{itemize}
Therefore, $\LL$-positions appear much more often than $\RR$-positions do. See Proposition \ref{2-sym} for more precise statement.

\section{Symmetric removable sets}
In this section, except for Proposition \ref{LRs1} and \ref{2-sym} and their corollaries, we consider general Ending Partizan Subtraction Nim without assuming $\LL\RR$-Subtraction Nim.

\begin{definition}[$p$-symmetric]
    Let $p\in\oZ$. Non-empty subset $S\subset \oZ$ is $p$-symmetric if for all $s\in S$, $p-s\in S$.
\end{definition}
Notice that $S$ is finite when $S$ is $p$-symmetric.

\begin{definition}[periodic]
    A sequence of outcomes $\oseq$ is periodic when there exist $A\in\pZ$ and $p\in \pZ$ such that for all $m\in \pZ$ with $m\ge A$, the equality $\OO_S(m+p)=\OO_S(m)$ holds. We say that $\seq$ is periodic with a period $p$. We call $\seq$ is purely periodic when $A=0$. Notice that if $p$ is a period, multiples of $p$ are also periods. In particular, our period is not necessarily the minimum period.
\end{definition}
\begin{proposition}
    In Ending Partizan Subtraction Nim, the sequence $\seq$ has a finite period if $S$ is finite.
\end{proposition}
\begin{proof}
    Assume $S = \{s_1, s_2, \dots, s_n\}$ where $s_1 < s_2 < \dots < s_n$. Since there are at most four possible outcomes for each position, the number of combinations of $a_n$ consecutive outcomes that can appear in the sequence of outcomes is finite. Therefore, by the Pigeonhole Principle, there exist positive integers $x$ and $y$ such that $s_n \le x < y$ and \[\OO_S(x-s_n)=\OO_S(y-s_n), \OO_S(x-s_n+1)=\OO_S(y-s_n+1),\dots, \OO_S(x-1)=\OO_S(y-1).\] The outcomes of the position obtained from the position with $x$ tokens is identical to that of the position obtained from the position with $y$ tokens, so it follows that $\OO_S(x) = \OO_S(y)$. Suppose that $\OO_S(x+k)=\OO_S(y+k)$ holds for all $k<t$. Then, by the induction hypothesis, $\OO_S(x+t-s_i)=\OO_S(y+t-s_i)$ holds for any $s_i\in S$. This implies that $\OO_S(x+t)=\OO_S(y+t)$. Thus, it is shown that the sequence of outcomes is periodic.
\end{proof}

\begin{proposition}\label{key lemma for p-sym}
    In Ending Partizan Subtraction Nim with a $p$-symmetric $S\subset \ZZ_{\ge 2}$, the following statements hold:
    \begin{enumerate}
        \item[(1)] $\OO_S(m)=\LL \implies \OO_S(m+p)\in \{\LL,\PP\}$.
        \item[(2)] $\OO_S(m)=\RR \implies \OO_S(m+p)\in \{\RR,\PP\}$.
        \item[(3)] $\OO_S(m)=\PP \implies \OO_S(m+p)=\PP$.
        \item[(4)] $\OO_S(m+p)=\NN \implies \OO_S(m)=\NN$.
    \end{enumerate}
\end{proposition}
\begin{proof}
    Regarding $(1)$, assume that $\OO_S(m)=\LL$. Starting from $m+p$ tokens, suppose that Right moves first. If Right remove $s\in S$ tokens, then Left can win by taking $p-s\in S$ tokens because $\OO_S(m)=\LL$. As Left can win from $m+p$ tokens as the second player, we have $\OO_S(m+p)\in\{\LL,\PP\}$. Similarly, we can prove $(2)$. Regarding $(3)$, assuming that $\OO_S(m)=\PP$, then $\OO_S(m+p)=\PP$ by $(1)$ and $(2)$. Regarding $(4)$, let $\OO_S(m)\neq \NN$. Then $\OO_S(m+p)\neq \NN$ by $(1),(2)$ and $(3)$.
\end{proof}

From the above, the following is derived.
\begin{proposition}
    In the case of $p$-symmetric, if the set of terminal positions contains at least one $\PP$-position, or if it contains both $\LL$-positions and $\RR$-positions, then the sequence of outcomes $\seq$ is not purely periodic with period $1$.
\end{proposition}

Furthermore, in the case where the set of terminal positions consists only of $\NN$-positions (under misère play), the sequence of outcomes is also not purely periodic with period $1$. Conversely, if it contains only $\LL$-positions or $\RR$-positions, $\seq$ is trivially purely periodic with period $1$. If the terminal positions consist only $\LL$-positions and $\RR$-positions, a period of $1$ is possible.

Curiously, for example, in the game with $S=\{2,3\}$ and under the rule that a $0$ position is an $\LL$-position and $1$ position is an $\NN$-position, then all positions where there are more than two tokens are $\LL$-positions. In contrast, in the game with $S=\{2,3\}$ and under the rule that a $0$ position is an $\NN$-position and $1$ position is an $\LL$-position, the sequence becomes purely periodic with period $5$, which repeats as $\NN\LL\PP\LL\NN$.

In the case where $S$ is $p$-symmetric, the following holds.
\begin{proposition}\label{LRs1}
    In $\LL\RR$-Subtraction Nim with a $p$-symmetric $S\subset \ZZ_{\ge 2}$, the period of the sequence $\seq$ is greater than or equal to $2$. In particular, $\seq$ does not have $1$-period, $\LL$.
\end{proposition}

\begin{proof}
    Lemma \ref{key lemma for p-sym} implies that $\OO_S(1+kp)\in\{\RR,\PP\}$ inductively on $k$.
\end{proof}

\begin{theorem}
    In Ending Partizan Subtraction Nim, if the set of removable numbers $S$ is $p$-symmetric, then the outcome sequence $\{\OO_S(n)\}_{n\in \pZ}$ has $p$ as a period.
\end{theorem}

\begin{proof}
    For each $i\in\{0,1,\dots, p-1\}$,
    \begin{enumerate}
        \item[(1)] when there exists $k\in\pZ$ such that $\OO_S(i+kp)=\PP$, by Lemma \ref{key lemma for p-sym} $(3)$, for all $l\in\pZ$ with $l\ge k$, $\OO_S(i+lp)=\PP$.
        \item[(2)] Except $(1)$, namely assume that $\OO_S(i+kp)$ is never $\PP$ for any $k\in \pZ$, and if $\OO_S(i+kp)=\LL$ for some $k$, then $\OO_S(i+lp)=\LL$ for all $l\ge k$ by Lemma \ref{key lemma for p-sym} $(1)$.
        \item[(3)] Except $(1)$, namely assume that $\OO_S(i+kp)$ is never $\PP$ for any $k\in \pZ$, and if $\OO_S(i+kp)=\RR$ for some $k$, then $\OO_S(i+lp)=\RR$ for all $l\ge k$ by Lemma \ref{key lemma for p-sym} $(2)$.
        \item[(4)] Except $(1)$, $(2)$ and $(3)$, namely if we assume that $\OO_S(i+kp)$ is never $\PP$, $\LL$ nor $\RR$, it implies that $\OO_S(i+lp)=\NN$ for all $l\in\pZ$. In this case, we set $k=0$.
    \end{enumerate}
\end{proof}

\begin{proposition}\label{2-sym}
    In $\LL\RR$-Subtraction Nim with $S\subset \ZZ_{\ge 2}$, if the minimum period of the sequence of outcomes $\seq$ is $2$, then the pair of outcomes that appears in the period is either $\{\LL,\RR\}$ or $\{\LL,\NN\}$. Furthermore, if $S$ is a finite set, then only the pair $\{\LL,\RR\}$ occurs in $\seq$. When the period is $\{\LL,\RR\}$, then $S\subset 2\ZZ_{\ge 1}$. Conversely, if $S\subset 2\ZZ_{\ge 1}$, then the outcomes are
    \begin{equation*}
        \OO_S(n)=
        \left\{
            \begin{array}{ll}
                \LL & \text{if } n \text{ is even}, \\
                \RR & \text{if } n\text{ is odd}. 
            \end{array}
        \right.
    \end{equation*}
\end{proposition}

\begin{proof}
    If $S\subset 2\oZ$, then any move preserves the parity of the number of tokens. Thus, 
    \begin{equation*}
        \OO_S(n)=
        \left\{
            \begin{array}{ll}
                \LL & \text{if } n \text{ is even}, \\
                \RR & \text{if } n\text{ is odd},
            \end{array}
        \right.
    \end{equation*}
    holds.

    Assume that the minimum period of the sequence $\seq$ is $2$. Theorem \ref{main theorem} implies that the only possible pairs of outcomes which appear periodically are
    \begin{enumerate}
        \item $\{\LL,\NN\}$,
        \item $\{\LL,\PP\}$,
        \item $\{\LL,\RR\}$.
    \end{enumerate}
    The periodicity condition implies the existence for $A\in\pZ$ such that for all $m\in\pZ$ with $m\ge A$, we have $\OO_S(m+2)=\OO_S(m)$. 

    We show by contradiction that the pair $\{\LL,\PP\}$ cannot occur. Assume for all $n$ with $n\ge A$, $\OO_S(n)\in \{\LL,\PP\}$. Let $n_1$ be such that $n_1\ge A+\min(S)$ and $\OO_S(n_1)=\PP$. By Proposition \ref{outcomes}, $\OO_S(n-\min(S))=\NN\not\in \{\LL,\PP\}$. Hence, a contradiction arises.

    Furthermore, we show that if $S$ is a finite set, only the pair $\{\LL,\RR\}$ occurs in $\seq$ by contradiction, and $S\subset 2\ZZ_{\ge 1}$. Assume for all $n$ with $n\ge A$, $\OO_S(n)\in \{\LL,\NN\}$. Let $n_2$ be such that $n_2\ge A+\max(S)$ and $\OO_S(n_2)=\NN$. By the assumption that only the pair $\{\LL,\NN\}$ occurs in the period, for all $s\in S$, $\OO_S(n_2-s)\in \{\LL,\NN\}$. This contradicts $\OO_S(n_2)=\NN$ and Proposition \ref{outcomes}, so the only possible pair of outcomes which appear periodically is $\{\LL,\RR\}$.
    
    Finally, we show that $S\subset 2\ZZ_{\ge 1}$ assuming that the period is $\{\LL,\RR\}$. Assume, for contradiction, that there exists an odd element $s_1\in S$. Then take $n_3\in\pZ$ such that $n_3\ge A+s_1$. Under this assumption, we have $\OO_S(n_3)\neq \OO_S(n_3-s_1)$. In particular, either $\OO_S(n_3)=\LL$ and $\OO_S(n_3-s_1)=\RR$, or $\OO_S(n_3)=\RR$ and $\OO_S(n_3-s_1)=\LL$. However, both cases contradict Proposition \ref{outcomes}. Therefore, $S\subset 2\ZZ_{\ge 1}$ holds.
\end{proof}

\begin{example}\label{counter}
    In $\LL\RR$-Subtraction Nim with $S=\{2k+3\mid k\in\pZ\}\cup \{4\}$, the outcomes are below:
    \begin{equation*}
        \OO_S(n)=
        \left\{
            \begin{array}{ll}
                \LL & \text{if }n=0,2,3 \text{ or }n\in\{2m+7\mid m\in\pZ\}, \\
                \RR & \text{if }n=1, \\
                \NN & \text{if }n=4,5 \text{ or }n\in\{2m+6\mid m\in\pZ\}.
            \end{array}
        \right.
    \end{equation*}
\end{example}

\section{Recent Results}
\subsection{Imbalance of Winning Strategies}
Given the same set $S$ of removable numbers, is there any relationship between the outcomes of Subtraction Nim in normal play and $\LL\RR$-Subtraction Nim? In general, such a relationship does not appear to hold. Although both games exhibit periodic behavior in their outcome sequences, their periods are often coprime. For a fixed removable set $S$, is there a meaningful relationship between the outcomes of Subtraction Nim in normal play and those of $\LL\RR$-Subtraction Nim? During a discussion in Combinatorial Game Theory Colloquium V, François Carret found the following Theorem.

\begin{theorem}
    Let $S$ be a non-empty subset of positive integers greater than or equal to $2$. Assume that whenever $n$ is $\RR$-position in $\LL\RR$-Subtraction Nim with $S$, $n$ is $\PP$-position in normal play. Then, the following are true:
    \begin{enumerate}
        \item If $n$ is $\PP$-position in $\LL\RR$-Subtraction Nim with $S$, $n$ is $\PP$-position too in normal play.
        \item If $n$ is $\NN$-position in $\LL\RR$-Subtraction Nim with $S$, $n$ is $\NN$-position too in normal play.
    \end{enumerate}
\end{theorem}
\begin{proof}
    Let $S\subset \ZZ_{\ge 2}$ be a non-empty subset. We write $\epsn(n)$ for the outcome of the position with $n$ tokens in $\LL\RR$-Subtraction Nim, and $\sn(n)$ for the outcome of the position with $n$ tokens in classical Subtraction Nim with the same removable set $S$.

    Assume that if the position with $n$ tokens is an $\RR$-position in $\LL\RR$-Subtraction Nim, then it is a $\PP$-position in classical Subtraction Nim.

    We prove the following two statements simultaneously by induction on $n$;
    \begin{enumerate}
        \item[(1)] $\epsn(n)=\PP$ implies $\sn(n)=\PP$.
        \item[(2)] $\epsn(n)=\NN$ implies $\sn(n)=\NN$.
    \end{enumerate}
    We assume $n\ge \min(S)$, since for $n<\min(S)$, by the rule of the game, no moves are available for either player and we have $\epsn(n)\in\{\LL,\RR\}$.

    Assume as the inductive hypothesis that both $(1)$ and $(2)$ hold for all non-negative integers less than $n$.

    (Proof of $(1)$ ) Assume $\epsn(n)=\PP$. Then, for all $s\in S$ with $n-s\ge 0$, we have $\epsn(n-s)=\NN$. By the inductive hypothesis $(2)$, this implies $\sn(n-s)=\NN$ for all such $s$. Therefore, $\sn(n)=\PP$.

    (Proof of $(2)$ ) Assume $\epsn(n)=\NN$. Then, there exists $s_1\in S$ such that $\epsn(n-s_1)\in\{\RR,\PP\}$. By the initial assumption and inductive hypothesis $(1)$, it follows that $\sn(n-s_1)=\PP$. Hence, $\sn(n)=\NN$.

    This completes the inductive proof of both $(1)$ and $(2)$.
\end{proof}
\subsection{Nowakowski's Normalization}
During a discussion in Combinatorial Game Theory Colloquium V, Richard Nowakowski made the following suggestion. $\LL\RR$-Subtraction Nim can be regarded as a combinatorial game under normal play by the following; Positions where no moves are available and the winner is Left are replaced by $1=\{0\mid\}$, and those where no move are available and the winner is Right are replaced by $-1=\{\mid 0\}$.

By those transformation, each position in $\LL\RR$-Subtraction Nim can be assigned not only an outcome class but also a game value in the sense of normal play.

In $\LL\RR$-Subtraction Nim, is the sequence of Nowakowski-normalized game values periodic for every finite subtraction set $S$ ? Our preliminary calculation shows that when $S=\{2,3\}$, the game values are periodic with the period $5$, but when $S=\{2,3,6\}$, the game values seem to be non-periodic, even additively.

To be precise, we write $[n]$ for the game value of $n$ tokens and $-[n]$ for its negative game value, $\OO_S([n+9]-[n])=\PP$ only for $n\equiv 0,2,4,5 \mod{9}$ except for $n=4$, and the second difference $\OO_S([28+9k]-2[19+9k]+[10+9k])$ is never $\PP$ as far as we computed.

\bibliography{sn-bibliography}%

\end{document}